\documentclass{article}
\usepackage{amssymb}
\usepackage{amsmath}

 \newcommand{\Iff}{\ \Longleftrightarrow\
} \newcommand{\Implies}{\ \Rightarrow\ }
\newcommand{\ov}{\overrightarrow} \newcommand{\ovd}{\ov{d}}
 \newcommand{\ovb}{\ov{b}}
\newcommand{\ovc}{\ov{c}} \newcommand{\ovg}{\ov{g}}
 \newcommand{\ovx}{\ov{x}}
 \newcommand{\A}{{\mathcal A}}
\newcommand{\B}{{\mathcal B}} \newcommand{\C}{{\mathcal C}}
\newcommand{\D}{{\mathcal D}} \newcommand{\F}{{\mathcal F}}
\newcommand{\G}{{\mathcal G}} \newcommand{\HH}{{\mathcal H}}
\newcommand{\la}{\langle} \newcommand{\ra}{\rangle}
  \newtheorem{theorem}{Theorem}[section]
\newtheorem{defn}[theorem]{Definition}
\newtheorem{lemma}[theorem]{Lemma} 
\newtheorem{cor}[theorem]{Corollary} 

\newtheorem{problem}[theorem]{Problem}
 \newenvironment{proof}
{\pagebreak[1]{\narrower\noindent {\bf
Proof:\quad\nopagebreak}}}{$\hfill\square$}

\begin{document}

\author{\textbf{Wesley Calvert} \\
Department of Mathematics \& Statistics \\
Murray State University\\
Murray, Kentucky 42071\\
wesley.calvert@murraystate.edu \and \textbf{Douglas Cenzer} \\
Department of Mathematics \\
University of Florida\\
Gainesville, FL 32611\\
cenzer@math.ufl.edu \and \textbf{Valentina S.\ Harizanov} \\
Department of Mathematics \\
George Washington University\\
Washington, DC 20052\\
harizanv@gwu.edu \and \textbf{Andrei Morozov} \\
Sobolev Institute of Mathematics\\ Novosibirsk, 630090, Russia\\
morozov@math.nsc.ru \thanks{
Calvert was partially supported by the NSF grants DMS-9970452,
DMS-0139626, and DMS-0353748, Harizanov by the NSF grants DMS-0502499
and DMS-0704256, and the last three authors by the NSF binational
grant DMS-0554841.}}

\title{Effective Categoricity of
Abelian $p$-Groups\\ } \maketitle

\begin{abstract}
We investigate effective categoricity of computable Abelian $p$-groups
$\mathcal{A}$. We prove that all computably categorical Abelian
$p$-groups are relatively computably categorical, that is, have
computably enumerable Scott families of existential formulas. We
investigate which computable Abelian $p$-groups are $\Delta _{2}^{0}$
categorical and relatively $\Delta^0_2$ categorical. 
\end{abstract}

\section{Introduction and Preliminaries}

In computable model theory we are interested in effective versions of
model theoretic notions and constructions. We consider in particular
computability theoretic bounds on the complexity of isomorphisms of
structures within the same isomorphism type. This paper is a sequel to
\cite{C-C-H-M} where we studied equivalence structures. Here we will
investigate computable Abelian groups. We consider only countable
structures for computable languages, and for infinite structures we
may assume that their universe is $\omega $. We identify sentences
with their G\"{o}del codes. The \emph{atomic} \emph{diagram} of a
structure $\mathcal{A}$ for $L$ is the set of all quantifier-free
sentences in $L_{A}$, $L$ expanded by constants for the elements in
$A$, which are true in $\mathcal{A}$. A structure is
\emph{computable} if its atomic diagram is computable. In other words,
a structure $\mathcal{A}$ is computable if there is an algorithm that
determines for every quantifier-free formula $\theta (x_{0},\ldots
,x_{n-1})$ and every sequence $(a_{0},\ldots ,a_{n-1})\in A^{n} $,
whether $\mathcal{A}\vDash \theta (a_{0},\ldots ,a_{n-1})$. 

A computable structure $\mathcal{A}$ is \emph{computably categorical}
if for every computable isomorphic copy $\mathcal{B}$ of
$\mathcal{A}$, there is a \emph{computable} isomorphism from
$\mathcal{A}$ onto $\mathcal{B}$. For example, the ordered set of
rational numbers is computably categorical, while the ordered set of
natural numbers is not.  Goncharov and Dzgoev \cite{GD}, and
Remmel \cite{R} proved that a computable linear ordering
is computably categorical if and only if it has only finitely many
successors. Goncharov and Dzgoev \cite{GD}, and Remmel
\cite{RBA} established that a computable Boolean algebra is computably
categorical if and only if it has finitely many atoms (see also
LaRoche \cite{L}). Miller \cite{Miller} proved that no computable tree
of height $\omega $ is computably categorical. Lempp, McCoy, Miller,
and Solomon \cite{L-M-M-S} characterized computable trees of finite
height that are computably categorical. Nurtazin \cite{N}, and
Metakides and Nerode \cite{MN} established that a computable
algebraically closed field of finite transcendence degree over its
prime field is computably categorical.  In the recent paper \cite{C-C-H-M}, the authors showed that a
computable equivalence structure $\A$ is computably categorical if and
only if $\A$ has at most finitely many finite equivalence classes, or
$\A$ has only finitely many infinite classes and there is a finite
bound on the size of the finite classes and there is at most one
finite $k$ such that $\A$ has infinitely many classes of size $k$.

The present paper will be concerned with the categoricity of Abelian
$p$-groups Goncharov \cite{G1} and Smith
\cite{Smi81} characterized computably categorical Abelian $p$-groups as
those that can be written in one of the following forms:
$(\mathbb{Z}(p^{\infty }))^{l}\oplus \mathcal{G}$ for $l\in \omega
\cup \{\infty \}$ and $\mathcal{F}$ is finite, or
$(\mathbb{Z}(p^{\infty }))^{n}\oplus \mathcal{F}\oplus
(\mathbb{Z}(p^{k}))^{\infty }$, where $n,k\in \omega $ and
$\mathcal{F}$ is finite. Goncharov, Lempp, and Solomon \cite{G-L-S}
proved that a computable, ordered, Abelian group is computably
categorical if and only if it has finite rank. Similarly, they showed
that a computable, ordered, Archimedean group is computably
categorical if and only if it has finite rank.

In \cite{C-C-H-M}, we characterized the relatively $\Delta _{2}^{0}$ categorical equivalence
structures as those with either finitely many infinite equivalence
classes, or with an upper bound on the size of the finite equivalence
classes. We also consider the complexity of isomorphisms for
structures $\mathcal{A}$ and $\mathcal{B}$ such that both
$Fin^{\mathcal{A}}$ and $Fin^{\mathcal{B}}$ are computable, or $\Delta
_{2}^{0}$. Finally, we show that every computable equivalence
structure is relatively $\Delta _{3}^{0}$ categorical.

For any computable ordinal $\alpha $, we say that a computable
structure $\mathcal{A}$ is $\Delta _{\alpha }^{0}$\emph{\ categorical}
if for every computable structure $\mathcal{B}$ isomorphic to
$\mathcal{A}$, there is a $\Delta _{\alpha }^{0}$ isomorphism form
$\mathcal{A}$ onto $\mathcal{B}$.  Lempp, McCoy, Miller, and Solomon
\cite{L-M-M-S} proved that for every $n\geq 1$, there is a computable
tree of finite height, which is $\Delta _{n+1}^{0}$ categorical but
not $\Delta _{n}^{0}$ categorical. We say that $\mathcal{A}$ is
\emph{relatively computably categorical} if for every structure
$\mathcal{B}$ isomorphic to $\mathcal{A}$, there is an isomorphism
that is computable relative to the atomic diagram of $\mathcal{B}$.
Similarly, a computable $\mathcal{A}$ is \emph{relatively }$\Delta
_{\alpha }^{0}$\emph{\ categorical} if for every $\mathcal{B}$
isomorphic to $\mathcal{A}$, there is an isomorphism that is $\Delta
_{\alpha }^{0}$ relative to the atomic diagram of
$\mathcal{B}$. Clearly, a relatively $\Delta _{\alpha }^{0}$
categorical structure is $\Delta _{\alpha }^{0}$ categorical. We are
especially interested in the case when $\alpha =2$.  McCoy
\cite{McCoy} \ characterized, under certain restrictions, $\Delta
_{2}^{0}$ categorical and relatively $\Delta _{2}^{0}$ categorical
linear orderings and Boolean algebras. For example, a computable
Boolean algebra is relatively $\Delta _{2}^{0}$ categorical if and
only if it can be expressed as a finite direct sum $c_{1}\vee \ldots
\vee c_{n}$, where each $c_{i}$ is either atomless, an atom, or a
$1$-atom. Using an enumeration result of Selivanov \cite{S}, Goncharov
\cite{G2} showed that there is a computable structure, which is
computably categorical but not relatively computably
categorical. 

Using a relativized version of Selivanov's enumeration
result, Goncharov, Harizanov, Knight, McCoy, Miller, and Solomon
\cite{G-H-K-M-M-S} showed that for each computable successor ordinal
$\alpha $, there is a computable structure, which is $\Delta _{\alpha
}^{0}$ categorical but not relatively $\Delta _{\alpha }^{0}$
categorical.  It was later shown by Chisholm, Fokina, Goncharov,
Harizanov, Knight, and Quinn \cite{cfghkm} that the same is true
for every computable limit ordinal $\alpha$.

It is not known whether for any
computable successor ordinal $\alpha $, there is a rigid computable
structure that is $\Delta _{\alpha }^{0}$ categorical but not
relatively $\Delta _{\alpha }^{0}$ categorical. Another open question
is whether every $\Delta _{1}^{1}$ categorical computable structure
must be relatively $\Delta _{1}^{1}$ categorical (see \cite{G-H-K-S}).

\label{sec:background}There are syntactic conditions that are equivalent to
relative $\Delta _{\alpha }^{0}$ categoricity. These conditions
involve the existence of certain families of formulas, that is,
certain Scott families.  Scott families come from Scott's Isomorphism
Theorem, which says that for a countable structure $\mathcal{A}$,
there is an $L_{\omega _{1}\omega }$ sentence whose countable models
are exactly the isomorphic copies of $\mathcal{A}$. A \emph{Scott
family }for a structure $\mathcal{A}$ is a countable family $\Phi $ of
$L_{\omega _{1}\omega }$ formulas, possibly with finitely many fixed
parameters from $A$, such that:

($i$) Each finite tuple in $\mathcal{A}$ satisfies some $\psi \in \Phi $;

($ii$)\ If $\overrightarrow{a}$, $\overrightarrow{b}$ are tuples in
$\mathcal{A}$ of the same length, satisfying the same formula in
$\Phi $, then they are automorphic; that is, there is an automorphism of $\mathcal{A}$ that maps
$\overrightarrow{a}$ to $\overrightarrow{b}$.

A \emph{formally c.e.\ Scott family} is a c.e.\ Scott family
consisting of finitary existential formulas. A \emph{formally $\Sigma
_{\alpha }^{0}$ Scott family} is a $\Sigma _{\alpha }^{0}$ Scott
family consisting of computable $\Sigma _{\alpha }$ formulas. Roughly
speaking, computable infinitary formulas are $L_{\omega _{1}\omega }$
formulas in which the infinite disjunctions and conjunctions are taken
over computably enumerable (c.e.) sets. We can classify computable
formulas according to their complexity as follows. A computable
$\Sigma _{0}$ or $\Pi _{0}$ formula is a finitary quantifier-free
formula. Let $\alpha >0$ be a computable ordinal. A computable $\Sigma
_{\alpha }$ formula is a c.e.\ disjunction of formulas $(\exists
\overrightarrow{u})\theta (\overrightarrow{x},\overrightarrow{u})$,
where $\theta $ is computable $\Pi _{\beta }$ for some $\beta <\alpha
$. A computable $\Pi _{\alpha }$ formula is a c.e.\ conjunction of
formulas $(\forall \overrightarrow{u})\theta
(\overrightarrow{x},\overrightarrow{u})$, where $\theta $ is
computable $\Sigma _{\beta }$ for some $\beta <\alpha $.  Precise
definition of computable infinitary formulas involves assigning
indices to the formulas, based on Kleene's system of ordinal notations
(see \cite{A-K}). The important property of these formulas is given in
the following theorem due to Ash.

\begin{theorem}
For a structure $\mathcal{A}$, if $\theta (\overrightarrow{x})$ is a
computable $\Sigma _{\alpha }$ formula, then the set
$\{\overrightarrow{a}: \mathcal{A}\models \theta
(\overrightarrow{a})\}$ is $\Sigma _{\alpha }^{0}$ relative to the
atomic diagram of $\mathcal{A}$.
\end{theorem}

\noindent An analogous result holds for computable $\Pi _{\alpha }$ formulas.

It is easy to see that if $\mathcal{A}$ has a formally c.e.\ Scott
family, then $\mathcal{A}$ is relatively computably categorical. In
general, if $\mathcal{A}$ has a formally $\Sigma _{\alpha }^{0}$ Scott
family, then $\mathcal{A}$ is relatively $\Delta _{\alpha }^{0}$
categorical. Goncharov \cite{G2} showed that if $\mathcal{A}$ is
$2$-decidable and computably categorical, then it has a formally c.e.\
Scott family. Ash \cite{A} showed that, under certain decidability
conditions on $\mathcal{A}$, if $\mathcal{A} $ is $\Delta _{\alpha
}^{0}$ categorical, then it has a formally $\Sigma _{\alpha }^{0}$
Scott family. For the relative notions, the decidability conditions
are not needed. Moreover, Ash, Knight, Manasse, and Slaman \cite
{A-K-M-S}, and independently Chisholm \cite{C} established the
following result.

\begin{theorem}
Let $\mathcal{A}$ be a computable structure. Then the following are
equivalent:

\begin{enumerate}
\item[$($a$)$] $\emph{A}$ is relatively $\Delta _{\alpha }^{0}$ categorical;

\item[$($b$)$] $\emph{A}$ has a formally $\Sigma _{\alpha }^{0}$ Scott
family;

\item[$($c$)$] $\emph{A}$ has a c.e.\ Scott family consisting of computable $\Sigma _{\alpha }$\ formulas.
\end{enumerate}
\end{theorem}

A structure is \emph{rigid} if it does not have nontrivial
automorphisms. A computable structure is $\Delta _{\alpha
}^{0}$\emph{\ stable} if every isomorphism from $\mathcal{A}$ onto a
computable structure is $\Delta _{\alpha }^{0}$. If a computable
structure is rigid and $\Delta _{\alpha }^{0}$ categorical, then it is
$\Delta _{\alpha }^{0}$ stable. A \emph{defining family} for a
structure $\mathcal{A}$ is a set $\Phi $ of formulas with one free
variable and a fixed finite tuple of parameters from $A$ such that:

($i$)\ Every element of $A$ satisfies some formula $\psi \in \Phi $;

($ii$)\ No formula of $\Phi $ is satisfied by more than one element of $A$.

\noindent A defining family $\Phi $ is \emph{formally $\Sigma _{\alpha}^{0}$}
if it is a $\Sigma _{\alpha }^{0}$ set of computable $\Sigma _{\alpha }$
formulas. In particular, a defining family $\Phi $ is \emph{formally c.e.}
if it is a c.e.\ set of finitary existential formulas. For a rigid
computable structure $\mathcal{A}$, there is a formally $\Sigma _{\alpha
}^{0}$ Scott family iff there is a formally $\Sigma _{\alpha }^{0}$ defining
family. 

In \cite{G2}, Goncharov obtained a rigid structure that is computably
stable but not relatively computably stable.  It is not known for any
computable ordinal $\alpha > 1$ whether there is a computable
structure that is $\Delta^0_\alpha$ stable but not relatively
$\Delta^0_\alpha$ stable.

In Section $2$, we investigate algorithmic properties of Abelian
groups and their characters, and we provide a connection
between equivalence structures and Abelian $p$-groups. In
Section $3$, we examine effective categoricity of Abelian
$p$-groups. We show that every computably categorical Abelian
$p$-group is also relatively computably categorical. 


The notions and notation of computability theory are standard and as in
Soare \cite{Soa87}. We fix $\left\langle \cdot ,\cdot \right\rangle $ to be
a computable bijection from $\omega ^{2}$ onto $\omega $. Let $(W_{e})_{e\in
\omega }$ be an effective enumeration of all c.e. sets.

\section{Computable Abelian $p$-Groups and Equivalence Structures}

Let $\mathcal{G} = (G, +^{\mathcal{A}}, 0)$ be a computable Abelian $p$-group,
and assume that $G = \omega$ and that $0$ is the identity for the
operation $+^{\mathcal{G}}$. It is immediate that the subtraction
function $-^{\mathcal{G}}$ as well as the inverse function are also
computable. In this section, we will focus on direct sums of cyclic
and quasicyclic groups and their connection with equivalence structures.

First, we need some definitions.  Let $p$ be a prime number. The group
$\G$ is said to be a $p$-group if, for all
$g \in G$, the order of $g$ is a power of $p$.  ${\mathbb Z}(p^n)$ is the cyclic
group of order $p^n$.  ${\mathbb Z}(p^{\infty})$ denotes the
quasicyclic $p$-group, the direct limit of the sequence ${\mathbb
Z}(p^n)$ and also the set of rationals in $[0,1)$ of the form $\frac
i{p^n}$ with addition modulo 1.  

\begin{defn} The \emph{period} of $G$ is $max\{|g|: g \in G\}$ if this
  quantity is finite, and $\infty$ otherwise.\end{defn}

The subgroups $p^{\alpha} G$, where $\alpha$ is an ordinal, are
defined recursively as follows:

\begin{description}
\item $p^0G = G$, $pG = \{px: x \in G\}$,

\item $p^{\alpha+1}G = p(p^\alpha G)$, and

\item $p^\lambda G = \bigcap_{\alpha<\lambda} p^\alpha G$ for limit
$\lambda$.
\end{description}

The \emph{length} of $\G$, $lh(\G)$, is the least ordinal $\alpha$
such that $p^{\alpha +1}G = p^{\alpha}G$. The \emph{divisible
part} of $\G$ is $D(\G) = p^{lh(G)}G$ and is a subgroup of $\G$. $\G$
is said to be \emph{reduced} if $D(\G) = \{0\}$. 
 
For an element $g \in G$, the height $ht(g)$ is $\infty$ if $g \in
D(G)$ and is otherwise the least $\alpha$ such that $g \notin
p^{\alpha+1}G$.  For a computable group $G$, $ht(g)$ can be an
arbitrary computable ordinal.  The height of $G$ is the supremum of
$\{ht(g): g \in G\}$.

Here are some classic results about Abelian $p$-groups which we will
need. The reader is referred to Fuchs \cite{F70} for a full
development of the theory of infinite Abelian groups.

\begin{theorem} \label{thm:cls} \begin{enumerate}
\item \textnormal{\textbf{(Baer)}} For any $p$-group $\G$, there
exists a subgroup $\A$ such that $G = \A \oplus D(\G)$.
\item \textnormal{\textbf{(Pr\"ufer)}} If $G$ is a countable Abelian $p$-group, then
$G$ is a direct sum of cyclic groups if and only if all nonzero
  elements have finite height.
\end{enumerate}
\end{theorem}

\begin{defn} Let $\A$ be a subgroup of $\G$. 
\begin{enumerate} 
\item $\A$ is a \emph{direct summand} of $\G$ if there exists
a subgroup $\B$ of $\G$ such that $\G = \A \oplus \B$.
\item $\A$ is a \emph{pure} subgroup of $\G$ if $A \cap p^n G = p^n A$
  for all $n$, that is the height of an element $a \in A$ is the same
  in $\A$ as it is in $\G$. 
\end{enumerate}
\end{defn}

We need some results from group theory on direct summands. See
\cite{F70} for more details.
\begin{theorem} \label{thm:ds}
\begin{enumerate}
\item \textnormal{\textbf{(Kulikov)}} If $\A$ has finite period and is
  a pure subgroup of $\G$, then $\A$ is a direct summand of $\G$.
\item \textnormal{\textbf{(Baer)}} Any divisible subgroup $\D$ of a group $\A$ is a direct
summand of $\A$.
\end{enumerate}
\end{theorem}

The \emph{Ulm subgroups} $G^{\alpha}$ are defined by $G^{\alpha} =
p^{\omega \alpha} G$. The $\alpha$th \emph{Ulm factor} $G_{\alpha}$
of $G$ is $G^{\alpha}/G^{\alpha+1}$, and the \emph{Ulm length} $\lambda(A)$
of $G$ is the least $\alpha$ such that $G^{\alpha} = G^{\alpha+1}$.

It follows from Theorem \ref{thm:cls} that each Ulm factor is a direct
sum of cyclic groups. Thus each Ulm factor $G_{\alpha}$ is a direct
sum of cyclic groups.  Now consider the sequence of \[P_\alpha(G) =
G_\alpha \cap \{x \in G : px = 0\}.\]  Let $u_\alpha(G) =
\dim_{\mathbb{Z}_p} P_{\alpha}(G) / P_{\alpha+1}(G)$.

\begin{theorem} [Ulm] Two Abelian $p$-groups $G$ and $H$ are isomorphic if and
only if they have the same Ulm sequence, that is, if and only if $\lambda(G)
= \lambda(H)$ and $u_{\alpha}(G) = u_\alpha(H)$ for all $\alpha$.
\end{theorem}

\begin{defn} \begin{enumerate}
\item $\oplus_{\alpha} \HH$ denotes the direct sum of $\alpha$ copies of
$\HH$ where $\alpha \leq \omega$. 

\item If $\A = \oplus_{i<\omega} Z(p^{n_i})$, then the \emph{character}
of $\A$ is 
\[
\chi(\A) = \{(n,k): card(\{i: n_i=n\}) \geq k\}.
\]
\item If $\G = \A \oplus \oplus_{\alpha} Z(p^{\infty})$ for some $\alpha
  \leq \omega$ and some $\A$ as above, then $\chi(\G) = \chi(\A)$. 
\item We say that $\G$ has {\rm bounded character} if for some finite $b$ and all
  $(n,k) \in \chi(\G)$, $n \leq b$, and is said to have
  \emph{unbounded character} otherwise.   
\end{enumerate}
\end{defn}

In the previous paper \cite{C-C-H-M}, we studied a similar notion for
equivalence structures, and constructed structures of various
characters.  We will show that for a general class of
such structures, a corresponding $p$-group \emph{with the same
character} may be constructed from a given equivalence structure.
Here are the basic definitions for computable equivalence structures.

An equivalence structure $\mathcal{A}=(A,E^{\mathcal{A}})$ consists of
a set $A$
with a binary relation $E^\mathcal{A}$ that is reflexive, symmetric, and transitive. An
equivalence structure $\mathcal{A}$ is \emph{computable} if $A$ is a
computable subset of $\omega $ and $E^\mathcal{A}$ is a computable relation. If $A$ is
an infinite set (which is usual), we may assume, without loss of generality,
that $A=\omega $. The $\mathcal{A}$-equivalence class of $a\in A$ is 
\begin{equation*}
\lbrack a]^{\mathcal{A}}=\{x\in A:xE^{\mathcal{A}}a\}\text{.}
\end{equation*}
We generally omit the superscript $^{\mathcal{A}}$ when it can be inferred
from the context.

\begin{defn}
\begin{enumerate}
\item[$($i$)$] Let $\mathcal{A}$ be an equivalence relation. The
\emph{character of} $\mathcal{A}$ is the set
\begin{equation*}
\chi (\mathcal{A})=\{\langle n,k\rangle :\ n,k>0\ \text{and}\ {\mathcal{A}}\ 
\text{has}\ \text{at least $k$ equivalence classes of size}\ n\}.
\end{equation*}

\item[$($ii$)$] We say that $\mathcal{A}$ has \emph{bounded} \emph{character}
if there is some finite $n$ such that all finite equivalence classes
of $\mathcal{A}$ have size at most $n$.
\end{enumerate}
\end{defn}

For both groups and equivalence relations, we may define the notion of
a \emph{character} as a subset $K$ of $(\omega - \{0\}) \times \omega$
such that for all $k$ and $n$, $(n,k+1) \in K \Implies (n,k) \in
K$. Let $o_{\G}(g)$ be the order of $g$ in $\G$. The $\G$ may be
omitted when it is clear. 

\begin{theorem} [Khisamiev \cite{Kh98}] 
Suppose that $\G$ is a computable Abelian $p$-group and
isomorphic to $\oplus_{\alpha} Z(p^{\infty}) \oplus \oplus_{i<\omega}
Z(p^{n_i})$. Then: \begin{enumerate}
\item $\{\la g,n \ra: o(g) = p^n\}$ is computable,
\item $\{\la g,n \ra: ht(g) \geq n\}$ is $\Sigma^0_1$, 
\item $D(\G)$ is a $\Pi^0_2$ set (recall that $D(\G)$ is the divisible
  part of $\G$), and
\item the character $\chi(\G)$ is a $\Sigma^0_2$ set.
\end{enumerate}
\end{theorem}

\begin{proof} (1) $o(g) = p^n \Iff p^n \cdot g = 0\ \&\ p^{n-1}\cdot g
  \neq 0$.

\medskip

(2) $ht(g) \geq p^n \Iff (\exists h) \left(p^n \cdot h = g\right)$.

\medskip

(3) Under the hypothesis, $ht(g) \geq \omega$ implies that $g \in
  D(\G)$, so that $g \in D(\G) \Iff (\forall n) \left(ht(g) \geq n\right)$. 

\medskip

(4) We have the following characterization of $\chi(\G)$ by Theorem
    \ref{thm:ds}: $\la n,k \ra \in \chi(\G)$ if and only if there
    exist $g_0,\dots,g_{k-1}$ such that for all $i<k$, $o(g) = p^n$
    and $ht(g) = 0$ and
\[
(*) (\forall c_0,c_1,\dots,c_{k-1} < p^n) [c_0 \cdot g_0 + \cdots +
    c_{k-1}\cdot g_{k-1} = 0 \Rightarrow (\forall i<k) (c_i=0)].
\]
That is, if $\la n,k \ra \in \chi(\G)$, then $\G$ has at least $k$
summands $\la g_1 \ra$, \dots, $\la g_{k-1} \ra$ isomorphic to $Z(p^n)$
and the sequence $g_1,\dots,g_{k-1}$ satisfies (*). On the other
hand, if $g_0,\dots,g_{k-1}$ satisfy (*), then they
generate a bounded, pure subgroup $\A$ of $\G$, which must be a summand
by Theorem \ref{thm:ds}. Hence $\G$ has at least $k$ summands of the
form $Z(p^n)$.
\end{proof}

Here is a connection between equivalence structures and Abelian $p$-groups.

\begin{theorem} \label{thm:etog} Let $p$ be a prime number and let $\A$ be a 
computable equivalence structure with character $K$ and with $\alpha$
infinite equivalence classes. We write $Inf(\A)$ to denote the set of
elements in $\A$ whose equivalence classes are infinite.  Then there
exists a computable $p$-group $\G$ isomorphic to \[\HH {\oplus}
\bigoplus\limits_{\alpha} Z(p^\infty),\] where $\HH$ is a direct sum of cyclic
$p$-groups with character $K$. Furthermore, if $Inf(\A)$ is
$\Sigma^0_1$, then $D(\G)$ is also $\Sigma^0_1$, and if $Inf(\A)$ is computable, then $D(\G)$ is also computable. \end{theorem}

\begin{proof} Let $\A$, $\alpha$ and $K$ be given as stated and let
  $\equiv$ denote $\equiv^{\mathcal A}$.  First define the computable set $B$
  of basic elements of $\G$ to consist of $\{a \in A: (\forall m<a)
  \neg (m \equiv a)\}$. Then enumerate $B$ in numerical order as $b_0,
  b_1, \dots$. The group $\G$ will be the direct sum of $\G_i$, where
  the groups $\G_i$ and $\G$ are constructed in stages $\G_i^s$.
  Initially, $\G_0^0 = \{0,1,\dots,p-1\}$ is a copy of $Z_p$ and in
  general, $\G_i^s$ is a copy of $Z(p^k)$, where $k = card(\{j<s: j
  \equiv b_i\})$.  Now $\G^s$ is isomorphic to the direct sum $\G_0^s
  \oplus \G_1^s \oplus \cdots \oplus G_s^s$.

At stage $s+1$, we initiate the component group $G_{s+1}^{s+1}$ and
 for each $i \leq s$, we check whether $s+1 \neq b_i$ and $s+1 \equiv
b_i$. If not, just let $\G_i^{s+1} = \G_i^s$. If so, then extend
$\G_i^s$ to a copy of $Z(p^{k+1})$, where $G_i^s$ is a copy of
$Z(p^k)$. That is, given that $G_i^s$ is a cyclic group of order $p^k$
with generator $a$, we put a new element $b$ into $G_i^{s+1}$  so that
$p \cdot b = a$, and also add elements to represent $i \cdot b + g$ for
$i = 1,2,\dots,p-1$ and $g \in G_i^s$. Then we also add elements to
 $\G$ to represent the new elements of $G_0^{s+1} \oplus \dots \oplus
 G_{s+1}^{s+1}$. This is done so that the elements of $\G^{s+1}$ are
 an initial segment of $\omega$. 

$\G$ will be a computable group, since for each $a \in \omega$, $a \in
 \G^a$ and for any two elements $a \leq b$, $a +^{\G} b$ is defined by
 stage $b$. 

It is clear that if $card([b_i])=n$ in $\A$, then for some $s$ and all
$t \geq s$, $\G_i^t$ is isomorphic to $Z(p^n)$, and if $card([b_i])=\omega$
in $\A$, then the inverse limit $\G_i$ of $\la G_i^s: s < \omega \ra$
  will be a copy of $Z(p^{\infty})$. Thus $\G$ has character $K$ and
  has $\alpha$ components of $Z(p^{\infty})$, as desired. 

For the furthermore clause, suppose that $Inf(\A)$ is a $\Sigma^0_1$
(respectively, computable) set. Then $\{i: b_i \in Inf(\A)\}$ is also
$\Sigma^0_1$ (respectively, computable). Now a sequence $\la c_0,\dots,c_{k-1}
\ra \in \G$ is in $D(\G)$ if and only if for all $i<k$, if $c_i \neq
0$, then $i \in Inf(\A)$.
\end{proof}

There are several corollaries to results from \cite{C-C-H-M}. 

\begin{cor} \label{cor1} For any $\Sigma _{2}^{0}$ character $K$, there is a computable
Abelian $p$-group $\G$ with character $K$ and with $D(\G)$ isomorphic
to $\oplus_{\omega} Z(p^{\infty})$. Furthermore, the domain of $D(\G)$ is a $\Sigma_{1}^{0}$ set.
\end{cor}

\begin{proof} By Lemma 2.3 of \cite{C-C-H-M}, there is a
  computable equivalence structure $\A$ with character $K$ such that
  $Inf(\A)$ is $\Sigma^0_1$.  The result now follows immediately from
  Theorem \ref{thm:etog}.
\end{proof}

\begin{cor} \label{cor2} For any $r\leq \omega $ and any bounded character $K$, 
there is a computable Abelian $p$-group $\G$ with character $K$ and with
$D(\G)$ isomorphic to $\oplus_r Z(p^{\infty})$. Furthermore, the
domain of
$D(\G)$ is a computable set.
\end{cor}

\begin{proof} By Lemma 2.4 of \cite{C-C-H-M}, there exists a
  computable equivalence structure $\A$ with character $K$, with
  exactly $r$ equivalence classes, and with $Inf(\A)$ computable. The
  result now follows from Theorem \ref{thm:etog}.
\end{proof}

\bigskip

If the character is not bounded, then the notions of an $s$-function
and an $s_{1}$-function are important. These functions were introduced
by Khisamiev in \cite{khis92}. The $s$-functions are called
\emph{limitwise monotonic} in \cite{KNS97}.

\begin{defn}
Let $f:\omega ^{2}\rightarrow \omega $. The function $f$ is an $s$-\emph{%
function} if the following hold:

\begin{enumerate}
\item For every $i$ and $s$, $f(i,s)\leq f(i,s+1)$;

\item For every $i$, the limit $m_i=lim_s f(i,s)$ exists.

\medskip

We say that $f$ is an $s_{1}$-\emph{function} if, in addition:

\item For every $i$, $m_{i}<m_{i+1}$.
\end{enumerate}
\end{defn}

The following result about characters, $s$-functions and
$s_1$-functions is immediate from Lemma 2.6 of \cite{C-C-H-M}.

\begin{lemma}
\label{l3} Let $\G$ be a computable Abelian $p$-group with infinite character
and with $D(\G)$ isomorphic to $\oplus_r Z(p^{\infty})$ where $r$ is
finite. Then 

\begin{enumerate}
\item There exists a computable $s$-function $f$ with corresponding
limits $m_{i}=lim_{s}f(i,s)$ such that $\langle n,k\rangle \in
\chi(\G)$ if and only if  $card(\{i:n=m_{i}\})\geq k.$

\item If the character is unbounded, then there is a
  computable $s_{1}$-function $f$ such that $\la m_i,1\ra \in
  \chi(\G)$ for all $i$.
\end{enumerate}
\end{lemma}

\begin{cor} \label{cor3} Let $K$ be a $\Sigma _{2}^{0}$ character, and let $r$ be finite.
\begin{enumerate}
\item Let $f$ be a computable $s$-function with the corresponding
limits $m_{i}=lim_{s}f(i,s)$ such that \[\langle k,n\rangle \in K \iff
card(\{i:n=m_{i}\})\geq k.\] Then there is a computable Abelian
$p$-group $\G$ with $\chi(\G)=K$ and with $D(\G)$ isomorphic to
$\oplus_r Z(p^{\infty})$.

\item Let $f$ be a computable $s_{1}$-function with corresponding
limits $m_{i}=lim_{s}f(i,s)$ such that $\langle m_{i},1\rangle \in K$ for
all $i$. Then there is a computable Abelian $p$-group $\G$ with 
$\chi (\G)=K$ and $D(\G)$ isomorphic to
$\oplus_r Z(p^{\infty})$.
\end{enumerate}
\end{cor}

\begin{proof}
These results follow from Theorem \ref{thm:etog} and from Lemma 2.8 of
\cite{C-C-H-M} where corresponding equivalence structures are
constructed.
\end{proof}

\section{Categoricity of Abelian $p$-Groups}

The computably categorical Abelian $p$-groups were characterized by
Goncharov \cite{G1} and Smith \cite{Smi81} as follows.

\begin{theorem} [Goncharov, Smith] \label{thm:smi} A computable Abelian $p$-group
$\G$ is computably categorical if and only if either
\begin{enumerate}
\item $\G \approx \oplus_{\alpha} Z(p^{\infty}) \oplus \mathcal{F}$,
  where $\alpha \leq \omega$, or
\item $\G \approx \oplus_r Z(p^{\infty}) \oplus \oplus_{\omega} Z(p^m)
  \oplus \mathcal{F}$, where $\mathcal{F}$ is a finite Abelian $p$-group and $r,m \in
  \omega$.  
\end{enumerate}
\end{theorem}

\begin{lemma} \label{lem:rcp} 
\begin{enumerate}
\item If computable groups $\G$  and $\HH$ are relatively
  $\Delta^0_{\alpha}$ categorical, and $\G$ and $\HH$ are $\Sigma^0_1$
  definable in $\G \oplus \HH$, then $\G \oplus \HH$ is relatively
  $\Delta^0_{\alpha}$ categorical.
\item If computable groups $\G_1, \G_2, \dots, \G_k$ are
  relatively computably categorical and each $\G_i$ is $\Sigma^0_1$
  definable in $\G = \G_1 \oplus \dots \oplus \G_n$, then $\G$ is relatively
  $\Delta^0_2$ categorical.
\end{enumerate}
\end{lemma}

\begin{proof} (1) The Scott formulas for $\G$ and $\HH$ may be modified for
  $\G \oplus \HH$ to quantify only over $\G$ and $\HH$. Then for
 an  element $a = g + h \in \G \oplus \HH$, the Scott formula is
\[
(\exists y \in G)(\exists z \in H)[x = y+z\ \&\ \phi^G(y) \ \&\ \psi^H(z)],
\]
where $\phi^G$ is the Scott formula for $g$, relativized to $\G$, and
$\psi^H$ is the Scott formula for $h$, relativized to $\HH$. It can be
checked that these formulas will be $\Sigma^0_{\alpha}$. For tuples
$\la a_1,\dots,a_n\ra$, the method is the same. If $\la
a_1,\dots,a_n\ra$ and $\la a_1',\dots,a_n'\ra$ satisfy the same Scott
formula, then we have $a_i = g_i + h_i$ and $a'_i = g'_i + h'_i$ where
$g_i$ and $g_i'$ satisfy the same Scott formula in $\G$ so that there
is an automorphism $\Phi$ of $\G$ taking $g_i$ to $g_i'$ and similarly
there is automorphism $\Psi$ of $\HH$ taking $h_i$ to $h_i'$. Then the
mapping taking $x + y$ to $\Phi(x) + \Psi(y)$ will be an automorphism
of $\G \oplus \HH$ taking each $a_i$ to $a_i'$. Therefore, these
formulas make up a Scott family as desired, so that $\G \oplus \HH$ is
relatively $\Delta^0_{\alpha}$ categorical.

(2) The relativized Scott formulas will now be $\Sigma^0_2$ and the
    proof follows as in part (1).
\end{proof}

We can now investigate the relative computable categoricity of computable
Abelian $p$-groups. 

\begin{theorem}
\label{thm:p1} If $\G$ is a computably categorical Abelian $p$-group, then
$\G$ is relatively
computably categorical.
\end{theorem}

\begin{proof} By Theorem \ref{thm:smi}, we have an expression for the
  form of $\G$.  Any finite structure is certainly relatively computably
  categorical, so we may ignore the $\F$ by Lemma \ref{lem:rcp}. 

(1) If all summands are of the form $Z(p^{\infty})$, then the Scott
  sentence for a tuple $\la g_1, \dots, g_n\ra$ simply tells the order
  $o_i$ of each $g_i$ and tells whether $c_1 \cdot g_1 + \cdots + c_n
  \cdot g_n = 0$ for all $c_1 < o_1, \dots, c_n < o_n$. Suppose that
  single elements $g$ and $g'$ have the same order $p^k$. Then there
  are divisible subgroups $\D$ and $\D'$ of $\G$ containing $g$ and
  $g'$ (respectively), each isomorphic to $Z(p^{\infty})$. Since $g$
  and $g'$ have the same order, there is an isomorphism taking $\D$ to
  $\D'$, which maps $g$ to $g'$ and, since $\D$ and $\D'$ are direct
  summands of $\G$ by Theorem \ref{thm:ds}, this can be extended to an
  isomorphism of $\G$ taking $g$ to $g'$.  This shows that groups of
  type (1) are relatively computably categorical.

(2) We may assume that $\F = 0$ and $\G = \D \oplus \HH$,
  where $\D = \oplus_r Z(p^{\infty})$ and $\HH = \oplus_{\omega}
  Z(p^m)$.  We claim that $\D$ is $\Delta^0_1$ definable, by the
  following. First, note that $g \in \D$ if and only if $g$ is
  divisible by $p^m$, so that $\D$ is $\Sigma^0_1$.  Second, note that
  there are exactly $p^{kr}$ elements in $\D$ of order $\leq p^k$.
  Now given $g \in G$, it follows that $g \in \D$ if and only if for any $p^{kr}$
  distinct elements of $\D$ with order at most $p^{kr}$, the element $g$ equals one of those elements. This
  gives a $\Pi^0_1$ formula for $\D$. The Scott formula for a single
  element $g \in \G$ says whether $c \cdot g \in \D$ for $c =
  1,p,\dots,p^n = o(g)$. Now suppose that $g_1$ and $g_2$ have the
  same Scott formula. If both are divisible, they are automorphic as
  in part (1). Now suppose that $g_1 \notin \D$ and $p^k
  \cdot g_1$ is not divisible for any $p^k < o(g_1)$. Then $g_i = d_i
  + h_i$ where $d_i \in \D$ and $h_i \in \HH$ with $o(h_i) = p^m$. By
  the previous argument, we may assume that $d_1 = d_2$. Now for each
  $i$, $h_i$ generates a pure subgroup $\HH_i$ of $\HH$ of order
  $p^m$, so by Theorem \ref{thm:ds}, we have $\HH = \HH_i \oplus \C_i$ for
  some (isomorphic) subgroups $C_i$ of $\HH$. There is certainly
  an isomorphism of $\HH_1$ onto $\HH_2$ taking $h_1$ to $h_2$ and
  this isomorphism may be extended to an automorphism of $\G$ taking $g_1$ to
  $g_2$. Now suppose that $p^k \cdot g_1$ is divisible for some $k$
  with $p^k < o(g_1)$. Then $g_1 = d_1 + h_1$, where $p^k \cdot h_1 =
  0$ and hence $h_1$ is divisible by $p^{m-k}$. Thus we can find
  $d_1'$ and $h_1'$ with $g_1' = d_1' + h_1'$ such that $g_1 = p^{m-k}(d_1' + h_1')$ and
  similarly for $g_2$ and $g_2'$. It follows from the previous
  argument that $g_1'$ and $g_2'$ are automorphic and the same
  automorphism takes $g_1$ to $g_2$. 

For a sequence $\la g_1,\dots,g_n\ra$ from $\G$, the Scott formula
includes the Scott formulas for each element and also says which
linear combinations $c_1 \cdot g_1 + \dots + c_n \cdot g_n = 0$ and
which are divisible, where each $c_i < o(g_i)$. We prove by induction
on $n$ that if $\la g_1, g_2, \dots, g_n \ra$ and $\la g_1', g_2',
\dots, g_n' \ra$ satisfy the same Scott formula, then they are
automorphic. The case $n=1$ is given above. For $n>1$, suppose that
$\la g_1, g_2, \dots , g_n \ra$ and $\la g_1', g_2', \dots , g_n' \ra$
satisfy the same Scott formula; it follows that $\la g_1, g_2, \dots,
g_{n-1} \ra$ and $\la g_1', g_2', \dots, g_{n-1}' \ra$ also satisfy the
same Scott formula, and are therefore automorphic by induction. There
are two cases. If some constant $c_i$ in a true equation is not
divisible by $p$, then, without loss of generality, we can solve the
equation for $g_n$ and use the observation above that $\la g_1, g_2,
\dots , g_{n-1} \ra$ and $\la g_1', g_2', \dots , g_{n-1}' \ra$ are
automorphic. If all constants of all true equations are divisible by
$p$, then we may find $a_i$ and $a_i'$ with $g_i = p a_i$ and $g_i' =
p a_i'$, and it suffices to show that $\la a_1,\dots,a_n\ra$ and
$\la a_1',\dots,a_n'\ra$ are automorphic. After some finite number of
divisions, we will eventually get coefficients not divisible by
$p$. \nopagebreak\end{proof}

\smallskip

Note that this argument depends on the fact that in $\oplus_{\infty}
Z(p^m)$ an element has order $\leq p^k$ if and only if it is divisible
by $p^{m-k}$. This is not true in the group $\oplus_{\infty}
Z(p^m) \oplus \oplus_{\infty} Z(p^n)$ where $m \neq n$. Of course any
group which is not computably categorical cannot be relatively
computably categorical, so Theorem \ref{thm:p1} characterizes the
relatively computably categorical Abelian $p$-groups. 

Next we consider $\Delta^0_2$ categoricity. The first case is when 
the reduced part of $\G$ has finite period.

\begin{theorem} \label{thm:finp}
\label{thm:p2} Suppose that $\G$ is isomorphic to $\oplus_{\alpha}
Z(p^{\infty}) \oplus \HH$, where $\HH$ has finite period and $\alpha
\leq \omega$. Then $\G$ is relatively $\Delta^0_2$ categorical.
\end{theorem}

\begin{proof} Since the period $p^r$ is finite, $\G$ is a direct sum of
  computably categorical groups of the form $\oplus_{\alpha}
  Z(p^{\infty})$ and $\oplus_{\omega} Z(p^m)$, together with some
  finite $F$. The Scott formulas are similar to those given above for
  the computably categorical groups, except that when $g$ is not
  divisible, we need to ask whether it is divisible by $p^k$ for each
  $k < r$, or is not divisible by $p^k$. The latter question is
  $\Pi^0_1$, so the Scott formulas are a conjunction of $\Sigma^0_1$
  and $\Pi^0_1$. Likewise for a sequence of elements, we need to ask
  whether each linear combination is divisible by $p^k$. After that,
  the argument is essentially the same as in Theorem \ref{thm:p1}. 
\end{proof}

There is a special case when $\G$ has no divisible part.

\begin{theorem} \label{thm:red} Suppose that $\G$ is a computable
  Abelian $p$-group with all elements of finite height. Then $\G$ is
  relatively $\Delta^0_2$ categorical.  [Note: These are exactly the
  reduced Abelian $p$-groups of length at most $\omega$.]
\end{theorem}

\begin{proof} For any finite subgroup $\F$ of $\G$ and any finite
  sequence $\ovg$ of elements of $\F$, the formula
  $\phi_{\ovg,\F}(\ovx)$ gives the atomic diagram of $\F[\ovg]$ and
  also states that $\F$ is a pure subgroup. The latter question is a
  $\Pi_1$ condition, 
\[
(\forall g\in F)(\forall n)(\forall x)[p^n\cdot x = g \Implies
  (\exists y \in F) p^n \cdot y = g].
\]
The Scott formula for $\ovg$ states that there exists a finite set $F
= \{a_1,\dots,a_t\}$ so that $\phi_{\ovg,\F}(\ovg)$ and furthermore,
no subgroup of $\F$ is pure. If $\ovg$ and $\ovg'$ satisfy the same
Scott formula, then there are isomorphic pure subgroups $\F$
containing $\ovg$ and $\F'$ containing $\ovg'$. Since $\F$ and $\F'$
are pure, there exist isomorphic summands $\HH$ and $\HH'$ so that $\G
= \F \oplus \HH = \F' \oplus \HH'$ so that the isomorphism between
$\F$ and $\F'$ may be extended to an automorphism of $\G$.

Now $\G = \oplus_{n<\omega} Z(p^{i_n})$ where each $i_n$ is finite, so that for
each $k$, $\oplus_{n<k} Z(p^{i_n})$ is a pure subgroup of $\G$ and any
finite sequence $\ovg$ will be included in one of these pure
subgroups. Thus every $\ovg$ satisfies some Scott formula. 
\end{proof}
 
We claim that no other Abelian $p$-groups are relatively $\Delta^0_2$
categorical. We first show this for groups which are products of
cyclic and quasicyclic groups. It turns out that even for a group $\G$
of infinite period with only finitely many $Z(p^{\infty})$ components,
$\G$ is not relatively $\Delta^0_2$ categorical. This differs from
equivalence structures, where any structure with only a finite number
of infinite equivalence classes is relatively $\Delta^0_2$
categorical. For equivalence structures, each class is necessarily
computable but $D(\G)$ need not be computable even when there is just
one copy of $Z(p^{\infty})$.

\begin{theorem} \label{thm:nrd} Suppose that a computable group $\G$ is isomorphic 
to $\oplus_{\alpha} Z(p^{\infty}) \oplus \HH$ for some group $\HH$
 with infinite period and all elements of finite height, where $\alpha
 \neq 0$. Then $\G$ is not relatively $\Delta^0_2$ categorical.
\end{theorem}

\begin{proof} Let $\G = \D \oplus \HH$, where $\D$ is divisible and
$\HH$ is a product of cyclic groups of unbounded order. Suppose $\G$
  had a $\Sigma^0_2$ family of Scott sentences. We will show that
  there is an element of the divisible part $\D$ whose Scott formula
  is satisfied by some element of $\HH$. But there can be no
  automorphism of $G$ mapping a divisible element to a non-divisible
  element. This contradiction will show that there is no such Scott
  sentence. 

We first assume that $\alpha = \omega$. 

Let $a$ be an element of $\D$ of order $p$ which satisfies a
$\Sigma^0_2$ Scott formula $\Psi(x,\ovd)$. We observe first that 
we may assume that the parameters $\ovd$ are independent, and in fact
belong to different components in the product decomposition of
$\G$. For the finite summands, we can assume the parameter is a generator, and
for the quasicyclic summands, we can take the parameter to have
maximal order (and therefore generate any other possible parameters). 

Of course any other divisible element of order $p$ must satisfy the
same formula, so we may assume that $a$ belongs to a subgroup $\A$
isomorphic to $Z(p^{\infty})$ which does not contain any of the
parameters. Then, by choosing witnesses $\ovc$ to instantiate the
existentially quantified variables in $\Psi$, we have a computable
$\Pi^0_1$ formula $\theta(x,\ovd,\ovc)$ satisfied by $a$.

We can now use the fact that this $\Pi^0_1$ sentence is true in $\G$
if and only if it is true in all finite subgroups of $\G$ containing
$a,\ov{c},\ovd$. 

Let $a,\ovc,\ovd$ generate a finite subgroup $\F_1$ of $\G$ and let
  $\mathcal{A}_1 = \A \cap \F_1$ be a finite subgroup of $\A$ of order $p^m$, and
  $\F_1 = \A_1 \oplus \mathcal{B}_1$ for some group $\B_1$.  Now find a factor
  group $\HH_1 \subset \HH$ of $\G$ of order type $\geq p^m$ and
  independent of $\F_1$; this exists since $\HH$ has infinite
  period. We may assume without loss of generality that $|\HH_1| =
  p^m$ and that each of $\ovd$ is in $\HH_1$. Let $\phi$ be an
  isomorphism from $\A_1 \oplus \F_1$ to $\HH_1 \oplus \F_1$ which is
  the identity on $\HH_1$, and $a' = \phi(a)$ and let $\ov{b}$ be
  the image of $\ov{c}$ under this mapping.

We claim that $\theta(a',\ovd,\ovb)$ holds. Now let $\HH'$ be any finite
  subgroup of $\G$ containing $a',\ovd,\ovb$; we may assume that $\HH'
  = \HH_1 \oplus \F_2$ where $\F_1 \subseteq \F_2$. Furthermore, we
  may assume (by taking an automorphism of $\G$ if necessary) that
  $\F_2 \cap \A_1 = \emptyset$. Then $\phi^{-1}$ may be extended to an
  isomorphism from $\HH'$ to a finite subgroup $\A_1 \oplus
  \F_2$. Since $\theta$ is $\Pi^0_1$, $\A_1 \oplus \F_2 \models
  \theta(a,\ovc,\ovd)$. Thus by the isomorphism, $\HH' \models
  \theta(a',\ovb,\ovd)$. Since this is true for any finite subgroup of
  $\G$, it follows that $\G \models \theta(a',\ovb,\ovd)$. Therefore
  $\Psi(a',\ovd)$ holds for the Scott formula $\Psi$. 

But $a'$ is not divisible, so it is not automorphic with $a$. This
contradiction proves the theorem in the first case. 

Suppose now that $\alpha$ is finite; 
we will assume for simplicity that $\alpha = 1$. 
Let $d_i$ be the parameter of greatest order in any quasicyclic
summand, and let $a$ be an element of that summand with $p \cdot a =
d$. Let $p^m$ be the order of $a$ and let $\F_1$ be the cyclic
subgroup generated by $a$; note that any other parameter in $\F_1$ is
a multiple of $d_i$. Now choose an element $g$ of order $p^m$,
generating a subgroup $\F_2$ so that any element from $\ovd$ in $\F_1
\oplus \F_2$ is already in $\F_1$. This can be done since $\HH$ has
infinite period. Now let $a' = a + p^{m-1} \cdot g$, so that $p\cdot a'
= d_i$. Then there is an automorphism of $\F_1 \oplus \F_2$ taking $a$
to $a'$ and preserving the parameters, defined by $\psi(j \cdot a + k
\cdot g) = j \cdot a' + k \cdot g$. This automorphism may be extended
to an automorphism of the finite subgroup $\HH_1$ generated by $a, \ovd,
\ovb$. Let $\ovc$ be the image of $\ovb$ under this extended automorphism.

We claim that $\theta(a', \ovd, \ovc)$ holds. Let $\HH'$ be any finite
subgroup of $\G$ containing $a', \ovd, \ovc$; we may assume that $\HH'
= \HH_1 \oplus \F$ for some $\F$, so that there is an automorphism of
$\HH'$ taking $a, \ovd, \ovb$ to $a', \ovd, \ovc$. Since $\HH'$ is a
finite subgroup of $\G$, we have $\HH' \vDash \theta(a, \ovd, \ovb)$
and hence, by the automorphism, $\HH' \vDash \theta(a', \ovd,
\ovc)$. Since this holds for any finite subgroup $\HH'$, it follows
that $\theta(a', \ovd, \ovc)$ and hence $\Psi(a', \ovd)$. But $a'$ is
not divisible, so cannot be automorphic with $a$. 
\end{proof}

In the paper \cite{C-C-H-M}, we defined a uniformly $\Sigma^0_2$
enumeration $K_e$ of the $\Sigma^0_2$ characters and an enumeration
$\C_e$ of the computable equivalence structures. 
For a total computable function $\phi_e: \omega \times \omega \to
\omega$, let $\G_e$ be the structure with universe $\omega$ and with
group operation $\phi_e$.

\begin{lemma}[\cite{C-C-H-M}]\label{lem:ke} For any fixed infinite
  $\Sigma^0_2$ character $K$, $\{e: K_e = K\}$ is $\Pi^0_3$
  complete.
\end{lemma}

\begin{theorem}[\cite{C-C-H-M}] \label{thm:ce} Let $\A$ be a
  computable equivalence structure with unbounded character $K$ and
  with infinitely many infinite equivalence classes. Suppose also that
  there exists a structure $\B$ with character $K$ and with no
  infinite equivalence classes. Then $\{e: C_e \simeq \A\}$ is
  $\Pi^0_4$ complete.
\end{theorem}

We can apply this analysis to $p$-groups for a similar result, using
Theorem \ref{thm:etog}.

\begin{theorem} \label{thm:ge} Let $\G$ be isomorphic to
  $\oplus_{\omega} Z(p^{\infty}) \oplus \HH$, with $\HH$ having
 infinite period and all elements of finite height.  Suppose also that
 there is a computable copy of $\HH$. Then $\{e: \mathcal{C}_e \simeq \G\}$ is
 $\Pi^0_4$ complete.\end{theorem}

\begin{proof} Fix such a group $\G$ with character $K$, 
and let $\C$ be an equivalence structure with character $K$.  It can be checked that $\{e: \G_e
 \simeq \G\}$ is a $\Pi^0_4$ set. For the completeness, we observe
 that the uniformity of the proof of Theorem \ref{thm:etog} provides a
 computable function $f$ such that $\mathcal{C}_a$ is isomorphic to $\mathcal{C}_b$ if and
 only if $\G_{f(a)}$ is isomorphic to $\G_{f(b)}$. Then $\mathcal{C}_e \simeq \mathcal{C}$
 if and only if $\G_{f(e)} \simeq \G$ and the completeness follows from
 Theorem \ref{thm:ce}.
\end{proof}

\smallskip

This gives the following result for categoricity.
 
\begin{theorem} \label{thm:ndg} Suppose that a computable group $\G$ is isomorphic 
to \[\oplus_{\omega}~Z(p^{\infty})~\oplus~\HH\] for some group $\HH$ with infinite period and all elements of
 finite height, and suppose, in addition, that there is a computable
 group isomorphic to $\HH$. Then $\G$ is not $\Delta^0_2$ categorical.
\end{theorem}

\begin{proof} If $\G$ were $\Delta^0_2$ categorical, then $\{e: \G_e \simeq
 \G\}$ has a $\Sigma^0_4$ definition. That is, let $M$ be a complete
 c.e. set, let $+$ be $+^{\G}$ and let $+_e$ be $+^{\G_e}$. Then $\G_e
 \simeq \G$ if and only if
\[
(\exists a)[a\in Tot^M \ \wedge\ (\forall m)(\forall n)(
\phi_a^{M}(m+n) = \phi^m_a(m) +_e \phi^M_a(n))].
\]
But this contradicts the $\Pi^0_4$ completeness from Theorem
\ref{thm:ge}. 
\end{proof}

\smallskip

Finally, all of the groups discussed above are certainly relatively
$\Delta^0_3$ categorical.

\begin{theorem} Let $\G$ be a computable group isomorphic to
  $\oplus_{\alpha} Z(p^{\infty}) \oplus \HH$, where $\HH$ has all
  elements of finite height. Then $\G$ is relatively $\Delta^0_3$
  categorical. 
\end{theorem}

\begin{proof} The divisible part $D(\G)$ can be defined by a $\Pi^0_2$
  sentence.  The definition of the Scott sentences builds on that of
  Theorem \ref{thm:red}. Given a finite pure subgroup $\F$ and a
  finite subgroup $\D$ of $D(\G)$, we define the formula
  $\phi_{\ovg,\F,\D}$ as before to give the atomic diagram of $\D
  \oplus \F [\ovg]$. Any such formula satisfied by $\ovg$ will be a
  Scott formula. 
\end{proof}

\smallskip

There is a stronger result for groups $\G$ with $D(\G)$ computable. 

\begin{theorem} \label{nt1} For any two isomorphic computable
Abelian $p$-groups $\G_1$ and $\G_2$ of length $\leq \omega$ such that
$D(\G_1)$ and $D(\G_2)$ are both computable, $\G_1$ and $\G_2$ are $\Delta
_{2}^{0}$ isomorphic.
\end{theorem}

\begin{proof} We will construct computable subgroups $\HH_1$ and $\HH_2$
  such that $\G_i = D(\G_i) \oplus \HH_i$. The subgroup $\HH_i$ will be defined as
  the union of a computable sequence $\A_s$ of pure finite
  subgroups. $\A_0 = \{0\}$. Given $\A_s$, find the least element $g
  \notin \A_s$ such that $\la \A_s \cup \{g\} \ra \cap D(\G_i) = \{0\}$
  and let $\A_{s+1} = \la \A_s \cup \{g\}\ra$. Then for each $s$,
  $D(\G_i) \cap \A_s = \{0\}$ and therefore $D(\G) \cap \HH_i =
  \{0\}$. The factor group $\HH_i$ is computable since $s \in \HH_i$ if and
  only if $s \in \A_{s+1}$. We argue by induction on the order of $g$
  that any element $g \in \G$ belongs to $D(\G) \oplus \HH_i$. For the
  initial case, suppose that $p \cdot g = 0$. Then either $g \in
  \A_{g+1}$ or else $a + g = d \neq 0$ for some $d \in D(\G)$. But in
  the latter case, $g = a - d \in D(\G) \oplus \HH_i$ as desired. Now
  suppose all elements of order $< p^m$ belong to $D(\G) \oplus \HH_i$
  and let $p^m g = 0$. Then $h = p \cdot g = a + d$ for some $a \in
  \HH_i$ and $d \in D(\G)$. Since $d$ is divisible, we can choose $d'$ so $d = pd'$. Then we have $a = p\cdot (g - d')$ so that, since $\HH_i$ is
  pure, $a = p \cdot a'$ for some $a' \in \HH_i$. Now $p \cdot (g - d' -
  a') = 0$, so that $g - d' - a' \in D(\G) \oplus \HH_i$ by the initial
  case.  Therefore, $g \in D(\G) \oplus \HH_i$ as desired.

It follows from Proposition \ref{thm:smi} that $D(\G_1)$ and $D(\G_2)$ are
computably isomorphic, and it follows from Theorem \ref{thm:red} that
$\HH_1$ and $\HH_2$ are $\Delta^0_2$ isomorphic. Now, the two
corresponding isomorphisms may be combined into a $\Delta_2^0$
isomorphism between $\G_1$ and $\G_2$.
\end{proof}

\section{Groups of length $> \omega$}

Index set results can tell us something about many of the groups of
greater length.  By using the calculations in Theorems 5.6, 5.15,
and 5.16 of \cite{idxsets}, we can prove the following.

\begin{theorem} Let $\G$ be an Abelian $p$-group.
\begin{enumerate}
\item If $\lambda(\G) = \omega \cdot n$ and $m \leq 2n-1$, then $\G$
  is not $\Delta^0_m$-categorical.
\item If $\lambda(\G) > \omega \cdot n$ and $m \leq 2n-2$, then $\G$
  is not $\Delta^0_m$-categorical.
\item If $\lambda(\G) = \omega \cdot n + k$ where $k \in \omega$, and
  $\HH$ is the reduced part of $\G$, the following hold:
\begin{enumerate}
\item If $\HH_{\omega n}$ is isomorphic to
  $\mathbb{Z}_{p^k}$ and $m \leq 2n-1$, then $\G$ is not
  $\Delta^0_m$-categorical.
\item If $\HH_{\omega n}$ is finite but not
  isomorphic to $\mathbb{Z}_{p^k}$ and $m \leq 2n-1$, then $\G$ is not
  $\Delta^0_m$-categorical.
\item If there is a unique $j < k$ such that $u_{\omega n + j}(\HH) =
  \infty$, and $m \leq 2n$, then $\G$ is not
  $\Delta^0_m$-categorical.
\item If there are distinct $i,j < k$ such that $u_{\omega n + i}(\HH)
  = u_{\omega n + j}(\HH) = \infty$ and $m \leq 2n+1$, then $\G$ is not
  $\Delta^0_m$-categorical.
\end{enumerate}
\end{enumerate}
\end{theorem}

\begin{cor} Let $\G$ be a computable Abelian $p$-group whose
  reduced part has a computable copy, and suppose that $\HH$ is the
  reduced part of $\G$.  Then if $\HH$ has infinitely
  many elements of height at least $\omega$, then $\G$ is not
  $\Delta^0_2$-categorical.\end{cor}

Results of Barker \cite{barker} give the following additional
information.

\begin{theorem}[Barker] Let $\G$ be a countable reduced Abelian $p$-group
  with computable Ulm invariants such that $\lambda(\G) = \omega \alpha
  + \omega + n$ and $\G_{\omega \alpha + \omega}$ is finite.  Then
  $\G$ is relatively $\Delta^0_{2 \alpha + 2}$-categorical but not $\Delta^0_{2
  \alpha + 1}$-categorical.\end{theorem}

Finally, we can prove the following result on relative categoricity:

\begin{theorem} Let $\G$ be a computable Abelian $p$-group with
  $\lambda(G) > \omega$ whose reduced part has no computable copy.
  Then $\G$ is not relatively $\Delta^0_2$-categorical. \end{theorem}

\begin{proof} The proof is the same as that of Theorem
  \ref{thm:nrd}. \end{proof}

\section{Open Problems}

The present paper does not completely characterize the relatively
$\Delta^0_2$-categorical or $\Delta^0_2$-categorical Abelian
$p$-groups.  We give below an exhaustive list of the open cases.

\begin{problem} Let $\G$ be a computable Abelian $p$-group isomorphic to
  $\mathcal{D} \oplus \mathcal{H}$, where $\mathcal{D}$ is a direct
  sum of finitely many copies of the Pr\"ufer group, and $\mathcal{H}$
  is reduced, with infinite period but all elements of finite height.
  Can $\G$ be $\Delta^0_2$-categorical? \end{problem}

\begin{problem} Let $\G$ be a computable Abelian $p$-group whose reduced
  part has no computable copy.  We have shown that $\G$ cannot be
  relatively $\Delta^0_2$-categorical.  Can it be
  $\Delta^0_2$-categorical? \end{problem}

\end{document}